\documentclass[12pt]{amsart}

\usepackage{amssymb, amsthm, amsmath}

\newtheorem{theorem}{Theorem}[section]

\theoremstyle{definition}
\newtheorem{definition}[theorem]{Definition}

\newtheorem{claim}{Claim}

\theoremstyle{remark}

\newcommand{\mcc}{{\raise 0.45ex \hbox{c}}}

\numberwithin{equation}{section}

\newcommand{\D}{\mathbb{D}}

\newcommand{\C}{\mathbb{C}}
\newcommand{\T}{\mathbb{T}}
\newcommand{\al}{\alpha}
\newcommand{\ip}[2]{\langle #1, #2 \rangle}
\newcommand{\mcH}{\mathcal{H}}
\newcommand{\mcB}{\mathcal{B}}

\title[Rational inner Agler class functions]{Rational inner functions in the Schur-Agler class of the polydisk}

\author{Greg Knese}

\address{University of California, Irvine, Irvine, CA 92697-3875}

\date{\today}

\email{gknese@uci.edu}


\thanks{This research was supported by NSF grant DMS-1001791}


\begin{document}
\bibliographystyle{apalike}
\maketitle

\begin{abstract}
Every two variable rational inner function on the bidisk has a special
representation called a transfer function realization.  It is well
known and related to important ideas in operator theory that this does
not extend to three or more variables on the polydisk.  We study the
class of rational inner functions on the polydisk which \emph{do}
possess a transfer function realization (the Schur-Agler class) and
investigate minimality in their representations.  Schur-Agler class
rational inner functions in three or more variables cannot be
represented in a way that is \emph{as} minimal as two variables might
suggest.
\end{abstract}

\section{Prologue}
Let $\D, \T, \D^n, \T^n$ denote the unit disk in $\C$, the unit
circle, the $n$-polydisk (or just polydisk), and the $n$-torus (or
just torus), respectively.
 
A rational inner function $f$ on the polydisk $\D^n$ is a rational
function:
\[
f = q/p \quad q, p \in \C[z_1,\dots, z_n]
\]
where $p$ has no zeros on $\D^n$, and an inner function:
\[
|f| = 1 \text{ a.e. on } \T^n.
\]
In one variable, rational inner functions are just the Blaschke
products:
\[
f(z) = \mu \prod_{j=1}^{N} \frac{z-a_j}{1-\bar{a_j}z} \quad a_j \in \D, \mu \in \partial \D.
\]
Rational inner functions on the polydisk, while not as powerful a tool
as Blaschke products, are still important because (1) they are dense
in the topology of local uniform convergence inside the set of
holomorphic functions on $\D^n$ with supremum norm at most one, and
(2) they are closely related to the study of stable polynomials,
polynomials whose roots do not intersect the polydisk.  It is our aim
to study a special class of rational inner functions, called the
Schur-Agler class rational inner functions, whose definition warrants
motivation.

Using matrices, Blaschke products can be represented in a way that
appears analogous to a linear fractional transformation.  For example,
\[
\frac{z^2 - 1/4}{1-(1/4)z^2} = A + zB(I-zD)^{-1}C
\]
where $A,B,C,D$ are the block entries of a $3\times 3$ unitary:
\[
U = \begin{matrix} & \begin{matrix} \C & \C^2 \end{matrix} \\
\begin{matrix} \C \\ \C^2 \end{matrix} & \begin{bmatrix} A & B \\ C &
  D \end{bmatrix} \end{matrix} = 
\begin{bmatrix} -1/4 & 0 & \sqrt{15}/4 \\
\sqrt{15}/4 & 0 & 1/4 \\ 
0 & 1 & 0 \end{bmatrix}
\]
The notation is supposed to indicate $A = -1/4, B = [0, \sqrt{15}/4],
C = [\sqrt{15}/4, 0]^{T}$, and $D = \begin{bmatrix} 0 & 1/4 \\ 1 &
  0 \end{bmatrix}$.  This type of representation is called a transfer
function realization (a term from engineering).

Two variable rational inner functions can be represented in a similar
way.  Take for example
\[
f(z_1,z_2) = \frac{2z_1z_2 - z_1 - z_2}{2 - z_1 - z_2}.
\]
If we let $U$ be the unitary matrix
\[
U = \begin{matrix} & \begin{matrix} \C & \C^2 \end{matrix} \\
\begin{matrix} \C \\ \C^2 \end{matrix} & \begin{bmatrix} A & B \\ C &
  D \end{bmatrix} \end{matrix} = 
\begin{bmatrix} 0 & \sqrt{2}/2 & \sqrt{2}/2 \\
\sqrt{2}/2 & 1/2 & -1/2 \\
\sqrt{2}/2 & -1/2 & 1/2 \end{bmatrix}
\]
and let
\[
E(z_1,z_2) = \begin{bmatrix} z_1 & 0 \\ 0 & z_2 \end{bmatrix}
\]
then writing $z = (z_1,z_2)$, it turns out
\begin{equation} \label{2varexample}
f(z) = A + B E(z) (I-DE(z))^{-1}C.
\end{equation}

Surprisingly, not all three variable rational inner functions have a
unitary transfer function realization.  This is known and related to
important ideas in operator theory.  What is also surprising---and one
of the main points of this article---is that even if a three variable
rational inner function has a transfer function realization, it cannot
always be represented in a way that is minimal for its degree.
Something we intend to show, is that the following rational inner
function
\[
g(z) = \frac{3z_1z_2z_3 - z_1z_2 - z_1z_3 - z_2 z_3}{3 - z_1 - z_2 -z_3}
\]
can be represented in the form
\[
g(z) = A + BE(z)(I-DE(z))^{-1}C
\]
where 
\[
U = \begin{matrix} & \begin{matrix} \C & \C^N \end{matrix} \\
\begin{matrix} \C \\ \C^N \end{matrix} & \begin{bmatrix} A & B \\ C &
  D \end{bmatrix} \end{matrix}
\]
is a block unitary matrix and $E(z)$ is an $N \times N$ diagonal
matrix with $z_1,z_2,z_3$ on the diagonal (in some combination).
Naively extrapolating from the previous two variable example, one
might expect that $N$ could be chosen to equal $N=3$.  This is not the
case. Instead, we show that $6 \leq N \leq 9$.

We now introduce the rest of the paper in a more general framework.

\section{Introduction}

Let us introduce three properties.

\begin{definition}
If $f :\D^n \to \D$ is holomorphic, then we say $f$ is \emph{satisfies
  the von Neumann inequality} or $f$ is in the \emph{Schur-Agler
  class} if
\[
||f(T)|| \leq 1
\]
for all commuting $n$-tuples of strict contractions $T=(T_1,\dots,
  T_n)$.  
\end{definition}

\begin{definition}
A function $f: \D^n \to \D$ possesses an \emph{Agler decomposition} if
there exist positive semi-definite kernels $K_j:\D^n \times \D^n \to
\C$, $j=1,\dots, n$, such that
\[
1-f(z)\overline{f(\zeta)} = \sum_{j=1}^{n} (1-z_j\bar{\zeta_j})
K_j(z,\zeta)
\]
\end{definition}

Recall that a function $K(z,\zeta)$ is positive semi-definite if for
every finite set $F$ the matrix
\[
(K(z,\zeta))_{z,\zeta \in F}
\]
is positive semi-definite.  (We would need an ordering to form an
actual matrix, but this is unimportant.)  For more information on
positive semi-definite kernels, refer to
\cite{AM02} Section 2.7.

\begin{definition} \label{def:trans} 
A function $f:\D^n \to \D$ has a \emph{transfer function realization}
if there is a Hilbert space decomposed into $n$ orthogonal summands
\[
\mcH = \mcH_1 \oplus \mcH_2 \oplus \dots \oplus \mcH_n
\]
and an isometric operator $V$ in
  $\mcB (\C \oplus \mcH)$ which we write as
\[
V = \begin{bmatrix} A & B \\ C & D \end{bmatrix}
\]
where $A \in \mcB (\C, \C), B \in \mcB (\mcH, \C), C \in \mcB (\C, \mcH),
D \in \mcB (\mcH, \mcH)$, such
that
\begin{equation} \label{realized}
f(z) = A + B E(z) (I - D E(z))^{-1} C
\end{equation}
where $E(z)$ is the diagonal matrix with block diagonal entries $z_1
I_{\mcH_1}$, $z_2 I_{\mcH_2}$, $\dots$, $z_n I_{\mcH_n}$.

When the Hilbert spaces are finite dimensional, we shall refer to the
size of the realization as 
\[
\sum_{j=1}^{n} \dim \mcH_j.
\]
\end{definition}

Note $\mcB(\mcH, \mathcal{K})$ represents the set of bounded linear
operators from Hilbert space $\mcH$ to Hilbert space $\mathcal{K}$.
Also, $\mcB(\mcH) := \mcB(\mcH, \mcH)$.

The connection between operator inequalities, positive semi-definite
decompositions, and realizations was made by J. Agler.

\begin{theorem}[\cite{jA88}] \label{aglerthm} 
Let $f: \D^n \to \D$ be holomorphic.  The following are equivalent:
\begin{enumerate}
\item $f$ satisfies a von Neumann inequality 
\item $f$ has an Agler decomposition
\item $f$ has a transfer function realization
\end{enumerate}
In particular, all three conditions are automatically true exactly
when $n=1$ or $2$ because of the following theorems.
\end{theorem}

\begin{theorem}[\cite{vN51}]
Every $f: \D \to \D$ holomorphic satisfies the von Neumann inequality.
\end{theorem}

\begin{theorem}[\cite{tA63}]
Every $f: \D^2 \to \D$ holomorphic satisfies the von Neumann
inequality.
\end{theorem}

\begin{theorem}[\cite{nV74} and \cite{CD75}]
Not every $f:\D^n \to \D$ holomorphic satisfies the von Neumann
inequality when $n >2$.
\end{theorem}

Accordingly, a function of $n$ variables that satisfies any of the
above properties will be called a Schur-Agler class function.  We
shall abbreviate this to just \emph{Agler class}.  The Agler class is
natural because of its interaction with operator theory and it is
possible to write down many examples of Agler class functions simply
by writing down a transfer function realization.  On the other hand,
it is difficult to determine whether a given function is in the Agler
class and it is difficult to write down Agler decompositions
explicitly (even in two variables).  For more general information on
Theorem \ref{aglerthm} see \cite{BT98} or the book \cite{AM02}.  For
more detailed information about the theorem see \cite{BB10}.

We are interested in \emph{rational} inner Agler class functions. Let
us state what holds in two variables.

\begin{theorem} \label{twovarfacts2} 
Let $f : \D^2 \to \D$ be a rational inner function and write $f = q/p$
with $q,p \in \C[z_1,z_2]$ of degree in $z_1$ at most $d_1$ and degree
in $z_2$ at most $d_2$.

\begin{enumerate}
\item A sums of squares decomposition holds.  There exist polynomials
  $A_1,\dots, A_{d_1}, B_1, \dots, B_{d_2} \in \C[z_1,z_2]$ such that
\[
|p(z)|^2 - |q(z)|^2 = (1-|z_1|^2) \sum_{j=1}^{d_1} |A_j(z)|^2 +
(1-|z_2|^2) \sum_{j=1}^{d_2} |B_j(z)|^2
\]

\item (\cite{aK89}) $f$ has a finite dimensional transfer function
  realization.  There exists a finite dimensional Hilbert space with
  direct sum decomposition $\mcH = \mcH_1 \oplus \mcH_2$ and a unitary
  matrix $U: \C \oplus \mcH \to \C \oplus \mcH$
\[
U = \begin{matrix} & \begin{matrix} \C & \mcH \end{matrix} \\
\begin{matrix} \C \\ \mcH \end{matrix} & \begin{bmatrix} A & B \\ C &
  D \end{bmatrix} \end{matrix}
\]
such that $\dim \mcH_j \leq d_j$ for $j=1,2$ and
\[
f(z) =  A + BE(z)(I-DE(z))^{-1}C
\]
where $E(z)$ is the block diagonal matrix:
\[
E(z) = \begin{bmatrix} z_1 I_{\mcH_1} & 0 \\
0 & z_2 I_{\mcH_2} \end{bmatrix}.
\]
\end{enumerate}

\end{theorem}

Unknown to most of the mathematics community, \cite{aK89} proved the
second item (which is well known to be equivalent to the first).  This
was pointed out to us by \cite{jB10}.  \cite{CW99} proved this result
using Agler's theorem (without concern for degree bounds) and showed
that the above result is essentially equivalent to And\^{o}'s
inequality.  For a direct proof of this result and more discussion see
\cite{BSV05} or \cite{gK08a}.

The fundamental question for this article is:
\begin{quote} To what extent does Theorem
\ref{twovarfacts2} carry over to $n$ variables if we stipulate that
our rational inner function is in the Agler class?
\end{quote}

The two variable arguments in \cite{CW99} can be used to establish the
following theorem.  A result of this type was announced by
\cite{jB10}.  We need to use some aspects of the proof so we sketch
the proof later on.

\begin{theorem} \label{unsurp} 
Let $f: \D^n \to \D$ be an Agler class rational inner function and
write $f = q/p$, with $q,p \in \C[z_1,\dots, z_n]$. Then,
\begin{enumerate}
\item A sums of squares decomposition holds.  There exist integers
  $N_1, \dots, N_n$ such that
\[
|p(z)|^2 - |q(z)|^2 = \sum_{j=1}^{n} (1-|z_j|^2) \sum_{k=1}^{N_j}
|A_{j,k}(z)|^2
\]
where $A_{j,k} \in \C[z_1,\dots, z_n]$.

\item $f$ has a finite dimensional transfer function realization.  
\end{enumerate}

\end{theorem}

As will be seen later, each $N_j$ is just the dimension of 
\[
\text{span} \{A_{j,k}: k=1,\dots, N_j\}
\]
and the dimensions of the Hilbert spaces in the transfer function
realization are also given by $N_j = \dim \mcH_j$.

One cannot control the number of terms in the sums of squares (and the
dimension of the transfer function realization) as precisely as in two
variables.  To emphasize this point, observe that if we write down a
finite dimensional transfer function realization as in Definition
\ref{def:trans}
\[
f(z_1,\dots,z_n) = A + BE(z)(I-DE(z))^{-1} C
\]
where we assume $\dim \mcH_j \leq d_j$, then $f = q/p$ is a rational
function where $q, p$ each have degree at most $d_j$ in the variable
$z_j$.  This follows from Cramer's rule.  (It can also be shown by
direct calculation that $f$ is indeed inner.)

Conversely, if one starts with an Agler class rational inner function
$f = q/p$ where $q,p$ each have degree at most $d_j$ in the variable
$z_j$, then something surprising occurs.  One cannot in general use
$\dim \mcH_j \leq d_j$ in the transfer function realization.  The
dimension of $\mcH_j$ may need to be chosen larger than $d_j$.
Theorem \ref{bounds} presents the bound we can prove on $\dim \mcH_j$
and Theorem \ref{example} gives an example which shows the bound $\dim
\mcH_j \leq d_j$ is not in general possible.

\begin{theorem} \label{bounds} 
Using the assumptions and notation of Theorem \ref{unsurp}, assume the
degree of $q, p$ is at most $d_j$ in the variable $z_j$ for $j=1,
\dots, n$.  Write $d =(d_1,d_2,\dots, d_n)$.  Then,
\begin{enumerate}
\item Each $A_{j,k}$ (from Theorem \ref{unsurp}) satisfies
\[
\deg_{z_i} A_{j,k} \leq \begin{cases} d_i & i \ne j \\ d_i-1 & i =
  j \end{cases}
\]
As a result, the integers $N_1,\dots, N_n$ in Theorem \ref{unsurp} can
be bounded as follows
\[
N_j \leq d_j \prod_{k\ne j} (d_k+1).
\]

\item The transfer function realization of $f$ can be chosen so that
  the dimensions of the blocks satisfy
\[
\dim \mcH_j \leq d_j \prod_{k\ne j} (d_k+1).
\]

In particular, $f$ has a transfer function realization of size 
\[
\sum_{j=1}^{n} d_j \prod_{k\ne j} (d_k+1).
\]
\end{enumerate}
\end{theorem}

\begin{theorem} \label{example} The rational inner function  
\[
f(z) = \frac{3z_1z_2z_3 - z_1- z_2 - z_3}{3 - z_1 - z_2 - z_3}
\]
is in the Agler class.  It has a transfer function realization of
size $9$ but it cannot be realized with size less than $6$.
\end{theorem}


\section{Proof of Theorems \ref{unsurp} and \ref{bounds}}

\begin{claim} If we have a sums of squares decomposition,
then we automatically have a finite dimensional transfer function
realization.
\end{claim}

\begin{proof}
  This is the well-known lurking isometry argument.  So, suppose
  $f=q/p$ is rational, inner, and Agler class, and
\[
|p(z)|^2 - |q(z)|^2 = \sum_{j=1}^{n} (1-|z_j|^2)|\vec{F}_j(z)|^2
\] 
where $\vec{F}_j \in \C^{N_j}[z]$ is a vector polynomial (the notation
is simpler if we use vector polynomials in place of sums of squares).

Rearranging we get
\[
|p(z)|^2 + \sum_{j=1}^{n} |z_j|^2 |\vec{F_j}(z)|^2 = |q(z)|^2 +
\sum_{j=1}^{n} |\vec{F_j}(z)|^2.
\]
By the polarization theorem for holomorphic functions
\[
\begin{aligned}
&p(z)\overline{p(\zeta)} +
  \sum_{j=1}^{n}\ip{z_j\vec{F_j(z)}}{\zeta_j \vec{F_j}(\zeta)} \\
&= q(z)\overline{q(\zeta)} +
  \sum_{j=1}^{n}\ip{\vec{F_j(z)}}{\vec{F_j}(\zeta)}.
\end{aligned}
\]

This formula can be used to show that the map which sends
\[
\begin{bmatrix} p(z) \\ z_1 \vec{F}_1(z) \\ \vdots \\ z_n
  \vec{F}_n(z) \end{bmatrix} \mapsto \begin{bmatrix} q(z) \\ 
  \vec{F}_1(z) \\ \vdots \\ \vec{F}_n(z) \end{bmatrix}
\]
is a well-defined linear and isometric map (initially defined on the
span of the elements of the form given on the left into the span of
the elements of the given form on the right).  It may be extended (if
necessary) to a unitary matrix $U$ of dimensions $1+\sum_{j=1}^{n}
N_j$ which we write in block form
\[
U = \begin{matrix} & \begin{matrix} \C & \C^N \end{matrix} \\
\begin{matrix} \C \\ \C^N \end{matrix} & \begin{bmatrix} A & B \\ C &
  D \end{bmatrix} \end{matrix}
\]
where $N = \sum_j N_j$.  Let us write 
\[
\vec{F}(z) = \begin{bmatrix} \vec{F}_1(z) \\ \vdots
  \\ \vec{F}_n(z) \end{bmatrix}
\]
and let $E(z)$ be the block $N\times N$ diagonal matrix with block
diagonal entries $z_1I_{N_1}, z_2I_{N_2},\dots, z_n I_{N_n}$.
Then, by construction of $U$
\[
\begin{aligned}
A p(z) + BE(z)\vec{F}(z) & = q(z)\\
C p(z) + DE(z) \vec{F}(z) &= \vec{F}(z).
\end{aligned}
\]
If one first solves for $\vec{F}(z)$ using the second equation, and
then inserts this into the first equation, we arrive at
\[
q/p(z) = A + BE(z)(I-DE(z))^{-1}C
\]
as desired.
\end{proof}

Next, we rehash the arguments of Cole and Wermer (which were
originally applied to two variables) in the $n$-variable context to
prove Theorem \ref{unsurp}.  This repetition is necessary because we
need some of the details of the proof in order to keep track of
degrees in Theorem \ref{bounds}. 

\begin{claim} \label{claim2}
Suppose $f=q/p$ is rational inner Agler class and let $r$ be the
maximum of the total degrees of $p$ and $q$.  Then $f$ has a sums of
squares decomposition:
\[
|p(z)|^2 - |q(z)|^2 = \sum_{j=1}^{n} (1-|z_j|^2)|\vec{F}_j(z)|^2
\] 
where each $\vec{F}_{j}$ is a vector polynomial of total degree less
than or equal to $r-1$.  Every such decomposition must satisfy this
degree bound.
\end{claim}

\begin{proof}
 By Agler's theorem, $f$ has an Agler decomposition:
\begin{equation} \label{aglerdecomp}
1-f(z)\overline{f(\zeta)} = \sum_{j=1}^{n} (1-z_j \bar{\zeta_j})
K_j(z,\zeta)
\end{equation}
where each $K_j$ is a positive semi-definite kernel. 

Since
\[
\begin{aligned}
\frac{1}{\prod_{j=1}^{n}(1-z_j\bar{\zeta_j})} &\geq 
\frac{1-f(z)\overline{f(\zeta)}}{\prod_{j=1}^{n}(1-z_j\bar{\zeta_j})}
\\
& \geq \frac{K_j(z,\zeta)}{\prod_{i \ne j} (1-z_i \bar{\zeta_i})} \\
& \geq K_j(z,\zeta)
\end{aligned}
\]
in the sense of positive semi-definite kernels (i.e. $K \geq L$ means
$K-L$ is positive semi-definite in this situation), it follows from
standard facts about reproducing kernels that each $K_j$ is the
reproducing kernel of a space of analytic functions and that for each
$j$ there is a Hilbert space $\mcH_j$ and an $\mcH_j$ valued analytic
function $\vec{F}_j: \D^n \to \mcH_j$ such that
\[
K_j(z,\zeta)  = \ip{\vec{F}(z)}{\vec{F}(\zeta)}.
\]

(See \cite{CW99} for more on the details of this
argument.)

Let us multiply equation \eqref{aglerdecomp} by
$p(z)\overline{p(\zeta)}$ and absorb this factor into the definition
of $\vec{F}_j(z)$ so that we really have
\[
p(z)\overline{p(\zeta)} - q(z)\overline{q(\zeta)} = \sum_{j=1}^{n}
(1-z_j\bar{\zeta_j}) \ip{\vec{F}_j(z)}{\vec{F}_j(\zeta)}.
\]
Now we let $z = \zeta = t \mu$ where $t \in \D$ and $\mu \in \T^n$:
\begin{equation} \label{tmu0}
\frac{|p(t\mu)|^2 - |q(t\mu)|^2}{1-|t|^2} = 
\sum_{j=1}^{n}|\vec{F}_j(t\mu)|^2
\end{equation}
The left hand side is a polynomial in $t, \bar{t}$ (because
$|p(\mu)|^2 = |q(\mu)|^2$) and a trigonometric polynomial in $\mu$.
Write
\[
\begin{aligned}
p(z) & = \sum_{\alpha} p_\alpha z^{\alpha} \\
q(z) & = \sum_{\alpha} q_{\alpha} z^{\alpha} \\
\vec{F}_j(z) &= \sum_{\alpha} \vec{F}_{j,\alpha} z^{\alpha}
\end{aligned}.
\]
(We are using multi index notation to write polynomials and power
series.)

Since 
\[
|p(t\mu)|^2 = \sum_{\alpha, \beta} p_{\alpha} \bar{p}_{\beta}
\mu^{\alpha -\beta} t^{|\alpha|} \bar{t}^{|\beta|}
\]
(and by performing similar computations for $|q(t\mu)|^2$ and
$|\vec{F}_{j}(t\mu)|^2$), we are able to compute the the zero-th
Fourier coefficient of \eqref{tmu0} when viewed as a Fourier series in
$\mu$:

\begin{equation} \label{tmu}
\frac{\sum_{\alpha} |t|^{2|\alpha|} (|p_{\alpha}|^2 -
  |q_{\alpha}|^2)}{1-|t|^2} = \sum_{j=1}^{n} \sum_{\alpha}
|\vec{F}_{j,\alpha}|^2 |t|^{2|\alpha|}.
\end{equation}

Recall $r$ denotes the maximum of the total degrees of $p$ and $q$.
Now, $|t|^{2}$ does not occur to any power larger than $r-1$ in
\eqref{tmu} and therefore
\[
|\vec{F}_{j,\alpha}|^2 = 0
\]
whenever $|\alpha| \geq r$.
  
This implies each $\vec{F}_j(z)$ is a Hilbert space valued
\emph{polynomial}.  It then follows that $|\vec{F}_j(z)|^2$ can be
replaced with the square of a vector polynomial.  One way to see this
is to observe that the coefficients of $z^\al\bar{z}^{\beta}$ in
$|\vec{F}_{j}(z)|^2$ form a finite dimensional positive semi-definite
matrix $X$, which when factored as $X = Y^*Y$ gives a representation
of $|\vec{F}_{j}(z)|^2$ as a vector polynomial squared.
\end{proof}

These two claims prove Theorem \ref{unsurp}.  To prove the bounds in
Theorem \ref{bounds}, we assume $p,q$ have multidegree at most $d =
(d_1,\dots,d_n)$.  Let $|d| = \sum_{j} d_j$, which is an upper bound
on the total degree of $p$ and $q$.

Consider again:
\[
p(z)\overline{p(\zeta)} - q(z)\overline{q(\zeta)} = \sum_{j=1}^{n}
(1-z_j\bar{\zeta_j}) \ip{\vec{F}_j(z)}{\vec{F}_j(\zeta)}.
\]
where we now know each $\vec{F}_j(z)$ must be a vector polynomial of
total degree at most $|d|-1$.  Let us focus on degree bounds for
$z_1$; our argument applies by symmetry to the other variables.

Let $M$ be a positive integer (which we use to amplify the degree of
$z_1$.) Replacing $z$ and $\zeta$ in the last equation with
$(z_1^M,z_2,\dots, z_n) = (z_1^M, z')$, we have
\begin{align}
&|p(z_1^M,z')|^2 - |q(z_1^M, z')|^2 \nonumber \\
&= (1-|z_1|^2)
(\sum_{j=0}^{M-1}|z_1|^{2j}) |\vec{F}_1(z_1^M,z')|^2
+ \sum_{j=2}^{n}
(1-|z_j|^2) |\vec{F}_j(z_1^M, z')|^2. \label{sosM}
\end{align}

We apply Claim \ref{claim2} to $p(z_1^M, z')$. Since the left hand
side has total degree at most $d_1M+d_2+\dots+d_n = d_1(M-1) + |d|$ in
$(z_1,z')$, the sums of squares polynomials on the right hand side
have total degree at most $d_1(M-1)+|d|-1$.

Suppose $z^\al$ has a nonzero coefficient in the Taylor expansion of
$\vec{F}_1$ and write $\al = (\al_1,\dots, \al_n)$.  Since
$|z_1^{M-1}\vec{F}_1(z_1^M,z')|^2$ appears as a sums of squares term
in \eqref{sosM}, our degree bound from Claim \ref{claim2} says
\[
M-1+M\al_1 + \sum_{j\geq 2} \al_j \leq d_1(M-1) + |d|-1
\]
and letting $M$ go to infinity we get $\al_1 \leq d_1 -1$.

Similarly, suppose $z^\al$ has a nonzero coefficient in the Taylor
expansion of $\vec{F}_j$, $j\ne 1$.  Then, looking at $\vec{F}_j$ in
\eqref{sosM}, our degree bound gives
\[
M\al_1 + \sum_{j\geq 2} \al_j  \leq d_1(M-1) + |d| - 1 .
\]
Letting $M$ go to infinity we get $\al_1 \leq d_1$.

The same argument applies to other variables.  This shows
$\vec{F}_{j}$ has multidegree at most $d - e_j$, with $e_j$ the
multi-index with $1$ in the $j$-th position and zeros elsewhere.

Therefore, $|\vec{F}_{j}(z)|^2$ is a reproducing kernel for a space of
polynomials of dimension at most 
\[
N_j = d_j \prod_{k \ne j}(d_k+1)
\]
and can therefore be written as the square of a vector polynomial with
at most $N_j$ components. (See the appendix of \cite{CW99} for some
background.)

This proves Theorem \ref{bounds}.

\section{Theorem \ref{example}: Three variable example}
The three variable rational inner function on the tridisk
$\mathbb{D}^3$
\[
f(z_1,z_2,z_3) = \frac{3z_1z_2z_3 - z_1z_2-z_2z_3-z_1z_3}{3-z_1-z_2-z_3}
\]
is in the Agler class because we can explicitly write an Agler
decomposition.

Namely, let
\[
S(z,w) = |P_1(z,w)|^2 + |P_2(z,w)|^2 + |P_3(z,w)|^2
\]
where
\[
\begin{aligned}
P_1(z,w) &= \sqrt{3}(zw-z/2-w/2) \\
P_2(z,w) &= \sqrt{3}(1-z/2-w/2) \\
P_3(z,w) &= (1/\sqrt{2})(z-w)
\end{aligned}
\]


Then, a decomposition for $f$ is given by
\[
\begin{aligned}
&|3-z_1-z_2-z_3|^2 - |3z_1z_2z_3 - z_1z_2-z_2z_3-z_1z_3|^2 \\
&= (1-|z_1|^2)S(z_2,z_3) + (1-|z_2|^2) S(z_1,z_3) + (1-|z_3|^2)
S(z_1,z_2).
\end{aligned}
\]

It remains to show that none of the sums of squares terms can be
chosen to be a single square.  So, suppose we have a decomposition
\[
\begin{aligned}
&|3-z_1-z_2-z_3|^2 - |3z_1z_2z_3 - z_1z_2-z_2z_3-z_1z_3|^2 \\
&= (1-|z_1|^2)SOS_1(z_2,z_3) + (1-|z_2|^2) SOS_2(z_1,z_3) + (1-|z_3|^2)
SOS_3(z_1,z_2).
\end{aligned}
\]
Each $SOS_j$ is a sum of squared moduli of polynomials. Note that by
Theorem \ref{bounds}, the squared polynomials in $SOS_1$ must have
multi-degree bounded by $(0,1,1)$ (with similar bounds for the other
sums of squares terms).

Setting $|z_2|=|z_3|=1$ yields
\[
|3-z_1-z_2-z_3|^2 - |3z_1z_2z_3 - z_1z_2-z_2z_3-z_1z_3|^2 =
(1-|z_1|^2)SOS_1(z_2,z_3)
\]
and $SOS_1(z_2,z_3)$ can be solved for explicitly when $z_2,z_3 \in
\T$.  Indeed, this term has to agree with $S(z_2,z_3)$ when $z_2,z_3
\in \T$:
\[
SOS_1(z,w) = 10 - 6\text{Re}(z+w) + 2\text{Re}(z\bar{w})
\]

We must show this is not a single square of a polynomial of degree
$(1,1)$ (on $\T^2$).  Supposing otherwise, we equate such an
expression
\[
\begin{aligned}
|a+bz+cw+dzw|^2 &= |a|^2+|b|^2+|c|^2+|d|^2 + \\
&+ 2\text{Re}\bar{a}(bz+cw+dzw)\\
& + 2\text{Re}(\bar{b}c\bar{z}w + \bar{b}dw + \bar{c}dz)
\end{aligned}
\]
with $SOS_1(z,w)$ and get the following by matching Fourier
coefficients
\begin{align}
10 &= |a|^2+|b|^2+|c|^2+|d|^2 \label{sos1}\\
-3 &= \bar{a}b + \bar{c}d  \label{sos2} \\
-3 &= \bar{a}c + \bar{b}d \label{sos3} \\
1 &= b\bar{c} \label{sos4} \\
0 &= \bar{a}d. \label{sos5}
\end{align}

These equations cannot all hold.  One of $a$ or $d$ equals zero by
\eqref{sos5} (but not both by \eqref{sos2}).  If $d = 0$, then $b=c
\in \T$ (by \eqref{sos2}, \eqref{sos3}, and \eqref{sos4}), and so
$\sqrt{8} = |a|$ (by \eqref{sos1}) contradicting equation
\eqref{sos2}:
\[
-3 = \bar{a} b.
\]
The case $a=0$ works the same.

Since the sums of squares terms must equal at least two squares, a
transfer function realization of $f$ has size at least $3*2 = 6$.  Our
explicit Agler decomposition shows $f$ has a realization of size
$3*3 = 9$.

\section*{Acknowledgments} 
I thank Joseph Ball and John M\mcc Carthy for their feedback on this
paper.

\bibliography{rifitsac}

\begin{thebibliography}{}

\bibitem[Agler, 1988]{jA88}
Agler, J. (1988).
\newblock Some interpolation theorems of {N}evanlinna-{P}ick type.
\newblock unpublished manuscript.

\bibitem[Agler and McCarthy, 2002]{AM02}
Agler, J. and McCarthy, J.~E. (2002).
\newblock {\em Pick interpolation and {H}ilbert function spaces}, volume~44 of
  {\em Graduate Studies in Mathematics}.
\newblock American Mathematical Society, Providence, RI.

\bibitem[And{\^o}, 1963]{tA63}
And{\^o}, T. (1963).
\newblock On a pair of commutative contractions.
\newblock {\em Acta Sci. Math. (Szeged)}, 24:88--90.

\bibitem[Ball, 2010]{jB10}
Ball, J.~A. (2010).
\newblock Transfer function realizations for inner functions on the bidisk.
\newblock Presentation at Southeastern Analysis Meeting at Georgia Tech.

\bibitem[Ball and Bolotnikov, 2010]{BB10}
Ball, J.~A. and Bolotnikov, V. (2010).
\newblock Canonical de {B}ranges-{R}ovnyak model transfer-function realization
  for multivariable {S}chur-class functions.
\newblock In {\em Hilbert spaces of analytic functions}, volume~51 of {\em CRM
  Proc. Lecture Notes}, pages 1--40. Amer. Math. Soc., Providence, RI.

\bibitem[Ball et~al., 2005]{BSV05}
Ball, J.~A., Sadosky, C., and Vinnikov, V. (2005).
\newblock Scattering systems with several evolutions and multidimensional
  input/state/output systems.
\newblock {\em Integral Equations Operator Theory}, 52(3):323--393.

\bibitem[Ball and Trent, 1998]{BT98}
Ball, J.~A. and Trent, T.~T. (1998).
\newblock Unitary colligations, reproducing kernel {H}ilbert spaces, and
  {N}evanlinna-{P}ick interpolation in several variables.
\newblock {\em J. Funct. Anal.}, 157(1):1--61.

\bibitem[Cole and Wermer, 1999]{CW99}
Cole, B.~J. and Wermer, J. (1999).
\newblock And{\^{o}}'s theorem and sums of squares.
\newblock {\em Indiana Univ. Math. J.}, 48(3):767--791.

\bibitem[Crabb and Davie, 1975]{CD75}
Crabb, M.~J. and Davie, A.~M. (1975).
\newblock von {N}eumann's inequality for {H}ilbert space operators.
\newblock {\em Bull. London Math. Soc.}, 7:49--50.

\bibitem[Knese, 2008]{gK08a}
Knese, G. (2008).
\newblock Bernstein-{S}zeg{\H o} measures on the two dimensional torus.
\newblock {\em Indiana Univ. Math. J.}, 57(3):1353--1376.

\bibitem[Kummert, 1989]{aK89}
Kummert, A. (1989).
\newblock Synthesis of two-dimensional lossless {$m$}-ports with prescribed
  scattering matrix.
\newblock {\em Circuits Systems Signal Process.}, 8(1):97--119.

\bibitem[Varopoulos, 1974]{nV74}
Varopoulos, N.~T. (1974).
\newblock On an inequality of von {N}eumann and an application of the metric
  theory of tensor products to operators theory.
\newblock {\em J. Functional Analysis}, 16:83--100.

\bibitem[von Neumann, 1951]{vN51}
von Neumann, J. (1951).
\newblock Eine {S}pektraltheorie f\"ur allgemeine {O}peratoren eines unit\"aren
  {R}aumes.
\newblock {\em Math. Nachr.}, 4:258--281.

\end{thebibliography}

\end{document}